\documentclass[reqno]{amsart}
\usepackage{amssymb}
\usepackage{upgreek}
\usepackage{xcolor}
\usepackage{tikz}
\newtheorem{theorem}{Theorem}[section]
\newtheorem{proposition}[theorem]{Proposition}
\newtheorem{lemma}[theorem]{Lemma}

\theoremstyle{definition}

\newtheorem{remark}[theorem]{Remark}

\numberwithin{equation}{section}
\allowdisplaybreaks
\begin{document}
\title[$s$-dimensional Kloosterman sum mod $p$]
{Fourth power mean of the general $s$-dimensional Kloosterman sum mod $p$}

\author{Nilanjan Bag}
\address{Ramakrishna Mission Vivekananda Educational and Research Institute, Belur Math, Howrah, West Benagal-711202, INDIA}
\curraddr{}
\email{nilanjanb2011@gmail.com}

\author{Anup Haldar}
\address{Ramakrishna Mission Vivekananda Educational and Research Institute, Belur Math, Howrah, West Benagal-711202, INDIA}
\curraddr{}
\email{anuphaldar1996@gmail.com}

\date{July 12, 2023}

\maketitle

\begin{abstract}
In this article, we prove an asymptotic formula for the fourth power mean of a general $s$-dimensional hyper-Kloosterman sum.  We find the number of solutions of certain congruence equations mod $p$ which play an integral part to prove our main result. We use estimates for character sums and analytic methods to prove our theorem.

\end{abstract}
\section{Introduction and statement of results}
In 1926, to study certain positive definite integral quadratic forms, Kloosterman \cite{kloostermann} introduced the exponential sum
\begin{equation}
 S(a,b;q)=\sum_{\substack{1\leq x\leq q\\ (x,q)=1}}e\left(\frac{ax+b\overline{x}}{q}\right),\nonumber 
\end{equation}
where $a$, $b$ and $q$ are arbitrary integers with $q\geq 1$. Here $e$ is defined as $e(y)=e^{2\pi iy}$ and $\overline{x}$ denotes the multiplicative inverse of $x$ mod $q$. Such sum is known as Kloosterman sum. Kloosterman had considerable interest in the order of magnitude of $ K(a,b;q)$. In his paper he proved that 
\begin{equation}
 S(a,b;q)=O(q^{3/4+\epsilon}(a,q)^{1/4})\hspace{1cm}(q\rightarrow\infty),\nonumber
\end{equation}
for every positive $\epsilon$.
There are various connections of this sums in number theory. Kloosterman in his study on cusp forms \cite{kloostermann2} showed that  any non trivial upper bound for $S(a,b;q)$ gives a corresponding improvement of Hecke's upper bound for the Fourier coefficients of 
certain cusp forms. There are numerous other applications of the order of magnitude of such sums is analytic number theory.
\par Further important example is hyper-Klooseterman sums. Hyper-Kloosterman sums were introduced by P. Deligne. These are higher dimensional generalization of  classical Kloosterman sum. Let $q\geq 3$ be a positive integer. For any fixed integer $s\geq 1$, the higher dimensional Kloosterman sum $K(m, s; q)$ is defined by
\begin{equation}
 K(m,s;q)=\sideset{}{'}\sum_{x_1=1}^{q}\cdots \sideset{}{'}\sum_{x_s=1}^{q} e\left(\frac{x_1+\cdots+x_s+m\overline{x_1\cdots x_s}}{q}\right)\nonumber
\end{equation}
and the general higher dimensional Kloosterman sum $K(m,s,\chi;q)$ is defined by
\begin{equation}
 K(m,s,\chi;q)=\sideset{}{'}\sum_{x_1=1}^{q}\cdots \sideset{}{'}\sum_{x_s=1}^{q}\chi(x_1\cdots x_s)
 \times e\left(\frac{x_1+\cdots+x_s+m\overline{x_1\cdots x_s}}{q}\right),\nonumber
\end{equation}
where $\displaystyle \sideset{}{'}\sum_{x=1}^{q}$ denotes the summation over all $1\leq x \leq q$ such that $\gcd(x,q)=1$, $m$ is any integer and $\chi$ is a Dirichlet character mod $q$. Hyper Kloosterman sums can be interpreted as inverse Mellin transform of powers of Gauss sums. Thus it is a very important quantity in the study of distribution of Gauss sums. As was denoted by Katz \cite{katz1}, Deligne's bound for Kloosterman sums implies that the set of normalized Gauss sums becomes equi-distributed on unit circle with respect to uniform probability Haar measure. 
\par Hyper-Kloosterman sums also occur in the theory of automorphic forms, for instance many has used the fact that powers of Gauss sums occur in the root number of functional equation of certain automorphic $L$-functions, Deligne bound and inverse Mellin transform property to obtain nontrivial estimates for the Langlands parameters of automorphic representations on $GL_n$. Also just as for classical Kloosterman sums, hyper-Kloosterman sums also occur in the spectral theory of $GL_k$ automorphic forms.
\par Many authors studied the arithmetical properties of $K(m,s; p)$, and obtained a series of interesting results. 
One of such results is due to Mordell \cite{mordell}. For odd prime $p$, he got the following estimate
\begin{equation}
 |K(m,s;p)|\ll p^{\frac{s+1}{2}}.\nonumber
\end{equation}
Later Deligne \cite{delign} improved Mordell's result and obtained the upper bound estimate
\begin{equation}\label{deligne}
 |K(m,s;p)|\leq (s+1)p^{\frac{s}{2}}.
\end{equation}
For many other important studies on such sums, see (\cite{luo}, \cite{shparlinski}, \cite{smith}, \cite{ye}, \cite{zhang}). It is well known that, for a principal character $\chi$,
 \begin{align*}
 K(m,1,\chi;p)=-2\sqrt{p}\cos(\theta(m)),
 \end{align*}
 where the angles $\theta(m)$ are equidistributed in $[0, \pi]$ with respect to the Sato-Tate measure $\frac{2}{\pi}\sin^2(\theta)d\theta$, for example, see \cite{katz}. Thus, moments can be estimated by evaluating the corresponding integral
 \begin{align*}
 \frac{1}{p-1}\sum_{m=1}^{p-1}|K(m,1,\chi;p)|^{2\ell}\approx 2^{2\ell}p^{\ell}\frac{2}{\pi}\int_{0}^{\pi}\cos^{2\ell}\theta\sin^{2}\theta~ d\theta,
 \end{align*}
 where $\ell$ is any positive integer. It would be interesting to investigate whether something similar is known for the higher dimensional generalized Kloosterman sums.
 \par 
 In this paper, we will concentrate 
 on fourth power mean value of the  general $s$-dimensional Kloosterman sum
\begin{equation}\label{k1}
 \sum_{m=1}^{p-1}\sum_{\chi ~\text{mod} ~p}\left|\sum_{x_1=1}^{p-1}\cdots \sum_{x_s=1}^{p-1}\chi(x_1\cdots x_s)\cdot 
 e\left(\frac{x_1+\cdots+x_s+m\overline{x_1\cdots x_s}}{p}\right)\right|^4.
\end{equation}
In case of $s=1$, it is easy to evaluate \eqref{k1}. It can be easily seen that 
\begin{align*}
\sum_{m=1}^{p-1}\sum_{\chi ~\text{mod} ~p}\left|\sum_{a=1}^{p-1}\chi(a)\cdot 
e\left(\frac{a+m\overline{a}}{p}\right)\right|^4
=2p^4-8p^3+10p^2-3p-1.
\end{align*}
For $s\geq 2$, Zhang and Li \cite{zhang} first studied this sum and obtained an exact computational formula for \eqref{k1} with $s=2$. For prime $p>3$, they proved the following identity 
\begin{align*}
&\sum_{m=1}^{p-1}\sum_{\chi ~\text{mod} ~p}\left|\sum_{a=1}^{p-1}\sum_{b=1}^{p-1}\chi(ab)\cdot 
e\left(\frac{a+b+m\overline{ab}}{p}\right)\right|^4\\
&=(p-1)(2p^5-7p^4+2p^3+8p^2+4p+1).
\end{align*}
Later, Zhang and Lv \cite{zhang-xing}
have obtained an asymptotic formula for $s=3$. For example, for $p>3$, they prove that 
\begin{equation}\label{zhang-lv}
 \sum_{m=1}^{p-1}\sum_{\chi ~\text{mod} ~p}\left|\sum_{a=1}^{p-1}\sum_{b=1}^{p-1}\sum_{c=1}^{p-1}\chi(abc)\cdot 
 e\left(\frac{a+b+c+m\overline{abc}}{p}\right)\right|^4=2p^8+O(p^{\frac{15}{2}}).
\end{equation}
In \cite{BB}, the first author and Barman improved the above result of Zhang and Lv by proving that the error term in the asymptotic formula \eqref{zhang-lv} is $O(p^7)$.
Finding an asymptotic formula for \eqref{k1} with $s=4$ seems to be more difficult as the idea used in \cite{zhang-xing} to derive \eqref{zhang-lv} 
is not sufficient to get a better estimate than Deligne bound. 
Later the first author and Barman \cite[Theorem 1]{BB1} used a result of P. Delign, which counts the number of $\mathbb{F}_p$ points on the surface 
\begin{align*}
(x-1)(y-1)(z-1)(1-xyz)-uxyz=0~, u\neq 0,
\end{align*}
 and then take average over $u$, which played an integral part to prove the asymptotic formula for fourth power moment of general $4$-dimensional Koosterman sums. To be specific, they proved the following theorem. For any prime $p>3$,
 \begin{align*}
 \sum_{m=1}^{p-1}\sum_{\chi~\text{mod}~p}\left|\sum_{x_1=1}^{p-1}\sum_{x_2=1}^{p-1}\sum_{x_3=1}^{p-1}\sum_{x_4=1}^{p-1}\chi(\prod_{i=1}^4x_i)
e\left(\frac{\sum_{i=1}^4x_i+m\prod_{i=1}^4\overline{x_i}}{p}\right)\right|^4
 =2p^{10}+O(p^9).
 \end{align*}
 In this paper, we generalize the previous results to get an asymptotic formula for $4$-th power mean values of the general $s$-dimensional hyper-Kloosterman sums and beat the trivial bound
 \begin{align*}
 \sum_{m=1}^{p-1}\sum_{\chi~\text{mod}~p}\left|\sum_{x_1=1}^{p-1}\cdots\sum_{x_s=1}^{p-1}\hspace{-.1cm}\chi(\prod_{i=1}^sx_i)e\hspace{-.1cm}
 \left(\frac{\sum_{i=1}^sx_i+m\prod_{i=1}^s\overline{x_i}}{p}\right)\right|^4\leq (s+1)^4p^{2s+2},
 \end{align*}
 which one can get using Deligne's estimate given in \eqref{deligne} for hyper-Kloosterman sums. To be specific, we prove the following,
 \begin{theorem}\label{t1}
 For any prime $p>3$ and any positive integer $s$, we have
  \begin{align*}
  &\sum_{m=1}^{p-1}\sum_{\chi ~\text{mod} ~p}\left|\sum_{x_1=1}^{p-1}\cdots \sum_{x_s=1}^{p-1}\chi(x_1\cdots x_s)\cdot 
 e\left(\frac{x_1+\cdots+x_s+m\overline{x_1\cdots x_s}}{p}\right)\right|^4\\
 &=2p^{2s+2}+O(p^{2s+3/2}).
 \end{align*}
 \end{theorem}
\section{ Few Notations}
 Let $p$ be a prime and $\mathbb{F}_p$ denote the finite field of $p$ elements and $\mathbb{F}_{p^m}$ be the finite of order $p^m$ over $\mathbb{F}_p$. Throughout the paper, $\displaystyle\sum_{X_n}$ stands for the sums over the $n$-tuples $(x_1,...,x_n)$, where $1\leq x_i\leq p-1$. Similarly, $\displaystyle\sum_{Y_n}$ stands for the sums over $y_i$'s. We have $f=O(g)$ or $f\ll g$ to denote $|f|\leq C\cdot g$ for some fix constant $C$. We denote the trivial character by $\chi_0$ which is defined as $\chi_0(n)=1$, if $\gcd(n,p)=1$; and $\chi_0(n)=0$, if $p|n$.
 The classical Jacobi sum is defined as,
 \begin{align}\label{jacobisum}
J(\chi,\psi)=\sum_{x=2}^{p-1}\chi(x)\psi(1-x),
 \end{align}
 where $\chi,~\psi$ belongs to $\hat{\mathbb{F}}_p$, the group of multiplicative characters on $\mathbb{F}_p$. For a multiplicative character $\chi$, the classical Gauss sum is defined as 
 \begin{align*}
     G(\chi)=-\sum_{x=1}^{p-1}\chi(x){\bf{e}}(x),
 \end{align*}
 where ${\bf{e}}(x)=\text{exp}({\frac{2\pi i x}{p}})$. It is well known that when $\chi\psi$ is primitive, we have \cite[Section 3.4]{iwa-kow}
 \begin{align}\label{Gauss-Jacobi}
J(\chi,\psi)=\frac{G(\chi)G(\psi)}{G(\chi\psi)}.
 \end{align}
 \section{Mixed moment of Gauss sums}
 In this section, we present a result of A. R- Le\'{o}n \cite{arl} which corresponds to the general distribution results of Gauss sums. All notations in this section is almost same as in \cite{arl}. Let $n\geq 1$ and ${{\bf{a}}_1,...,{\bf{a}}_n}$ be fixed non-zero $r$-tuples in $\mathbb{Z}^r$. Consider $\eta_1,...,\eta_n:\mathbb{F}_p^{\times}\rightarrow\mathbb{C}^{\times}$ be $n$ multiplicative characters. For every $m$, let $T_m$ be the set of multiplicative characters $\chi$ of $(\mathbb{F}_{p^m}^{\times})^r$. Here 
 \begin{align*}
 \chi=(\chi_1,...,\chi_r),
 \end{align*}
 where $\chi_i:\mathbb{F}_{p^m}^{\times}\rightarrow\mathbb{C}^{\times}$.
Let $S_m $ be the subset of $T_m$ consisting of $\chi$ such that
 \begin{align*}
 \eta_i\chi^{{\bf{a}}_1}:=\eta_i\chi_{i}^{a_{i1}}\cdots\chi_{r}^{a_{ir}}\neq \chi_0,
\end{align*} 
for $i=1,...,n$. For any character $\chi:\mathbb{F}_{p^m}^{\times}\rightarrow\mathbb{C}^{\times}$, denote the corresponding Gauss sum over $\mathbb{F}_{p^m}$ by $G_m(\chi)$. For every $\chi\in S_m$, the element $\Phi_m(\chi)\in (S^1)^n$ is defined as 
\begin{align}\label{phi_m(x)}
\Phi_m(\chi)=\left(p^{-m/2}\chi({{\bf{t}}_1})G_m(\eta_1\chi^{{{\bf{a}}_1}}),...,p^{-m/2}\chi_n({{\bf{t}}_n})G_m(\eta_n\chi^{{{\bf{a}}_n}}))\right),
\end{align}
where ${{\bf{t}}_1,...,{\bf{t}}_n}\in (\mathbb{F}_p^{\times})^r$ and  $\chi({\bf{t}})$ is given by 
\begin{align*}
\chi({\bf{t}})=\prod_{l=1}^{r}\chi_{l}(t_l).
\end{align*}
 Define the map
\begin{align*}
\Lambda_{\bf{c}}:{\bf{t}}=(t_1,...,t_n)\mapsto {\bf{t^c}}= t_1^{c_1}\cdots t_n^{c_n},
\end{align*}
for some $n$-tuples $(c_1,...,c_n)\in\mathbb{Z}^n.$ Take
\begin{align}\label{Sigma}
\Sigma_m(\Lambda_{\bf{c}})=|S_m|^{-1}\sum_{\chi\in S_m}\Lambda(\Phi_m(\chi)).
\end{align}
Then in a recent work, A. R-Le\'{o}n proved the following result, 
\begin{proposition}\label{prop1}\cite[Proposition 1]{arl}
Let $a=\sum_{i}\min_{j:a_{ij}=0}|a_{ij}|$. There exists a constant $A(c)$ such that, for every $m>\log_p(1+a)$,
\begin{align*}
\Sigma_m(\Lambda_{\bf{c}})\leq \frac{A(c)(p^m-1)^rp^{-m/2}+a(p^m-1)^{r-1}}{(p^m-1)^{r-1}(p^m-1-a)}.
\end{align*}
\end{proposition}


\section{ multivaribale congruences modulo $p$}
In this section, we obtain number of solutions for  certain multivariable congruences modulo $p$, which play an integral part in proving lemmas in Section 4.
\begin{lemma}\label{lemma 3.1}
    Let $p$ be any prime and $s$ be a positive integer. Let $\mathcal{A}(s)$ be the cardinality of the set 
    \begin{align*}
        \left\{(x_1,x_2,...x_s)\in \mathbb{F}_p^{s}|~x_1\cdots x_s\equiv 1 \bmod p, ~2\leq x_i \leq p-1\right\}.
    \end{align*}
    Then we have 
    \begin{align*}
       \mathcal{A}(s) =\begin{cases}
           \frac{(p-2)^{s}+(p-2)}{p-1}~\text{if }~ $s$ ~\text{is even};\\
           \frac{(p-2)^{s}-(p-2)}{p-1}~\text{if }~ $s$~ \text{is odd}.
       \end{cases}
    \end{align*}
\end{lemma}
\begin{proof}
    We prove the lemma using induction. It is trivially true for $s=1$. Let for any positive integer $n\leq s$, we have 
    \begin{align}\label{A(n)}
        \mathcal{A}(n-1)=\frac{(p-2)^{n-1}+(-1)^{n-1}(p-2)}{p-1}.
    \end{align}
    Now we use iteration over $s$ for $A(s)$ to get the expression,
    \begin{align*}
        \mathcal{A}(s)=(p-1)^{s-1}-\binom{s}{1}\mathcal{A}(s-1)-\binom{s}{2}\mathcal{A}(s-2)-\cdots-\binom{s}{s-2}\mathcal{A}(2).
    \end{align*}
    Now using \eqref{A(n)} in the above expression we get 
\begin{align*}
    \mathcal{A}(s)&=(p-1)^{s-1}-\frac{1}{p-1}\left[\binom{s}{1}(p-2)^{s-1}+\binom{s}{2}(p-2)^{s-2}+\cdots+\binom{s}{s-2}(p-2)^2\right]\\
    &-\frac{p-2}{p-1}\left[\binom{s}{1}(-1)^{s-1}+\binom{s}{2}(-1)^{s-2}+\cdots+\binom{s}{s-2}(-1)^2\right]\\
    &=(p-1)^{s-1}-\frac{1}{p-1}\left[(p-1)^s-(p-2)^s-s(p-2)-1\right]\\
    &-\frac{p-2}{p-1}\left[0-(-1)^s-s(-1)-1\right]\\
    &=\frac{(p-2)^s+(-1)^s(p-2)}{p-1},
\end{align*}
which completes the proof of the lemma. 

\end{proof}

In the next lemma we replace
\begin{align*}
x_1\cdots x_s\equiv 1 \bmod p,
\end{align*} 
by 
\begin{align*}
x_1\cdots x_s\equiv u \bmod p,
\end{align*}
where $u\neq 1,0$ and we prove the following,
\begin{lemma}\label{lemma 3.2}
    Let $p$ be a prime and $s$ be any positive integer. Let $\mathcal{A}_u(s)$ be the cardinality of the set 
    \begin{align*}
\left\{(x_1,x_2,...,x_s)\in\mathbb{F}_p^{s}|~x_1\cdots x_s=u \bmod p,~ 2\leq x_i\leq p-1\right\},
    \end{align*}
     where $u\neq 0,1.$ Then we have
     \begin{align*}
         \mathcal{A}_u(s)=\frac{(p-2)^s-(-1)^s}{p-1}.
     \end{align*}
\end{lemma}
\begin{proof}
    We prove this lemma using the formula for $\mathcal{A}(s)$ and induction of $s$. For $s=1$, it is trivial. Let the statement be true for $s-1$. Then we can write $\mathcal{A}_u(s)$ as 
    \begin{align*}
        \mathcal{A}_u(s)=\mathcal{A}(s-1)+(p-3)\mathcal{A}_u(s-1).
    \end{align*}
    Hence from the induction hypothesis and Lemma \ref{lemma 3.1}, we deduce
    \begin{align*}
        \mathcal{A}_u(s)&= \frac{(p-2)^{s-1}+(-1)^{s-1}(p-2)}{p-1}+(p-3)\frac{(p-2)^s-(-1)^s}{p-1}\\
        &=\frac{(p-2)^{s}-(-1)^{s}}{p-1}.
    \end{align*}
    
\end{proof}
Next, we consider more than one congruences modulo $p$ and calculate the number of simultaneous solutions.
\begin{lemma}\label{lemma S(s)}
Let $p$ be an odd prime and $\mathcal{S}(s)$ be the cardinality of the set  
\begin{eqnarray*}
 \{(x_1,...x_{s+1},y_1,...,y_{s})\in\mathbb{F}_p^{2s+1}|~x_1\cdots\ x_{s+1}\equiv y_1\cdots y_{s}\bmod p,\\ 
 ~~\prod_{i=1}^{s}(x_i-1)\equiv \prod_{i=1}^{s}(y_i-1) \bmod p,1\leq x_i\leq p-1, 1\leq y_i\leq p-1\}.
\end{eqnarray*}
Then we have
\begin{align*}
\mathcal{S}(s)= \left[(p-1)^s-(p-2)^s\right]^2+(p-1)(p-2)^{2s-2}+\frac{(p-2)^{2s-2}(3-2p)+p-2}{p-1}.
\end{align*}
\end{lemma}
\begin{proof}
    For $0\leq u\leq p-1$, let $N(u)$ be the number of solutions of 
    \begin{align*}
        \prod_{i=1}^{s}(x_i-1)\equiv u \bmod p.
    \end{align*}
        where $\leq x_1\leq p-1$. Then for $u=0$ we get 
        \begin{align*}
            N(0)=\binom{s}{1}(p-2)^{s-1}+\binom{s}{2}(p-2)^{s-2}+\cdots+\binom{s}{s}
            =(p-1)^{s}-(p-2)^{s}.
        \end{align*}
    For $u=p-1$, one can observe that 
    \begin{align*}
        N(p-1)=(p-2)^{s-1}-\mathcal{A}(s-1).
    \end{align*}
    Similarly, for $1\leq u< p-1$, we have 
    \begin{align*}
        N(u)=(p-2)^{s-1}-\mathcal{A}_{u}(s-1).
    \end{align*}
    Now for any choice of $x_1,...,x_s$ and $y_1,...,y_s$; $x_{s+1}$ is uniquely determined by 
    \begin{align*}
        x_1\cdots\ x_{s+1}\equiv y_1\cdots y_{s} \bmod p.
    \end{align*}
    Hence we get 
    \begin{align}\label{S(s)}
        \mathcal{S}(s)&=\sum_{u=0}^{p-1}N(u)^2\notag\\&=\left[(p-1)^s-(p-2)^s\right]^2\notag\\&+\left[(p-2)^{s-1}-\mathcal{A}(s-1)\right]^2+(p-2)\left[(p-2)^{s-1}-\mathcal{A}_u(s-1)\right]^2\notag\\
        &=(p-2)^{2s}+(p-1)^{2s}-2(p-1)^s(p-2)^s+(p-2)^{2s-1}+(p-2)^{2s-2}\notag\\
        &+(p-2)\mathcal{A}_u(s-1)^2-2(p-2)^s\mathcal{A}_u(s-1)+\mathcal{A}(s-1)^2-2(p-2)^{s-1}\mathcal{A}(s-1).
    \end{align}
    Using Lemma \ref{lemma 3.1} and \eqref{lemma 3.2} we get 
    \begin{align}\label{S(s)1}
    (p-2)\mathcal{A}_u(s-1)^2&=\frac{p-2}{(p-1)^2}\left[(p-2)^{2s-2}+1-2(2-p)^{s-1}\right],
    \end{align}
    \begin{align}\label{S(s)2}
        (p-2)^s\mathcal{A}_u(s-1)&=\frac{1}{p-1}\left[(p-2)^{2s-1}+(2-p)^s\right],
        \end{align}
        \begin{align}\label{S(s)3}
    \mathcal{A}(s-1)^2&=\frac{1}{(p-1)^2}\left[(p-2)^{2s-2}+(p-2)^2-2(2-p)^{s}\right],
    \end{align}
    \begin{align}\label{S(s)4}
         (p-2)^{s-1}\mathcal{A}(s-1)&=\frac{1}{p-1}\left[(p-2)^{2s-2}-(2-p)^{s}\right].
    \end{align}
    Adding both side of \eqref{S(s)1} and \eqref{S(s)3} we have
    \begin{align*}
         (p-2)\mathcal{A}_u(s-1)^2+ \mathcal{A}(s-1)^2=\frac{(p-2)^{2s-1}+(p-2)^{2s-2}+(p-2)^2+(p-2)}{(p-1)^2}.
    \end{align*}
    Similarly, adding both side of \eqref{S(s)2} and \eqref{S(s)4} we have
    \begin{align*}
        (p-2)^s\mathcal{A}_u(s-1)+(p-2)^{s-1}\mathcal{A}(s-1)=\frac{(p-2)^{2s-1}+(p-2)^{2s-2}}{p-1}.
    \end{align*}
    Now putting the above two expressions in \eqref{S(s)} we conclude Lemma \ref{lemma S(s)} . 
\end{proof}
\begin{lemma}\label{lemma T(s)}
Let $p$ be an odd prime, $s$ be a positive integer and $\mathcal{T}(s)$ be the cardinality of the set  
\begin{eqnarray*}
 \{(x_1,...,x_{s},y_1,...,y_{s})\in\mathbb{F}_p^{2s}|~x_1\cdots\ x_{s}\equiv y_1\cdots y_{s}\bmod p,
 \\ \prod_{i=1}^{s}(x_i-1)\equiv \prod_{i=1}^{s}(y_i-1)\bmod p,1\leq x_i,y_i\leq p-1\}.
\end{eqnarray*}
Then we have
\begin{align*}
\mathcal{T}(s)&=(p-1)^{2s-1}-2(p-2)\left((p-1)^{s-1}(p-2)^{s-1}-p^{s-1}\right)+f(p,s),
\end{align*}
where 
\begin{align*}
    f(p,s)=\frac{p(p-2)^2((p-2)^{2s-2}-p^{s-1})}{(p-1)^2}-\frac{(p^2-4)(p^{s-1}-1)}{(p-1)^2}.
\end{align*}
\end{lemma}
\begin{proof}
First we want to divide the $\mathbb{F}_p$-points on the given surface into parts. We follow the same technique as given in \cite[Lemma 5]{BB1}. In our case, we consider the bijection $x_s\rightarrow x_sy_s$ to get the form for $\mathcal{T}(s+1)$ as
    \begin{align}\label{july17}
        \mathcal{T}(s+1)=(p-1)^{2s}+p\mathcal{R}(s)-2\mathcal{S}(s),
    \end{align}
    where $\mathcal{S}(s)$ is same as in the last lemma and $\mathcal{R}(s)$ is the cardinality of the set
    \begin{align*}
    \{(x_1,...,x_{s+1},y_1,...,y_{s})\in\mathbb{F}_p^{2s+1}|~\prod_{i=1}^{s+1}x_i\equiv \prod_{i=1}^{s}y_i \bmod p, \prod_{i=1}^{s+1}(x_i-1)\equiv 0\bmod p, \\ \prod_{i=1}^{s}(x_i-1)\equiv \prod_{i=1}^{s}(y_i-1)\bmod p, 1\leq x_i\leq p-1, 1\leq y_i\leq p-1\}.
    \end{align*}
    We have 
    \begin{align*}
        \mathcal{R}(s)=\mathcal{T}(s)+\mathcal{R}'(s),
    \end{align*}
    where $\mathcal{R}'(s)$ is the cardinality of the set 
    \begin{align*}
         \{(x_1,...,x_{s+1},y_1,...,y_{s})\in\mathbb{F}_p^{2s+1}|\prod_{i=1}^{s+1}x_i\equiv \prod_{i=1}^{s}y_i\bmod p, \\ \prod_{i=1}^{s}(x_i-1)\equiv \prod_{i=1}^{s}(y_i-1)\equiv 0\bmod p,\\ 1\leq x_i\leq p-1,~ i\neq {s+1},~ 2\leq x_{s+1}\leq p-1,~ 1\leq y_i\leq p-1\}.
    \end{align*}
    Now we find an exact computational formula for $\mathcal{R}'(s)$. Here 
    \begin{align*}
        \mathcal{R}'(s)&=\binom{s}{1}\left[\binom{s}{1}\mathcal{A}(2s-1)+\binom{s}{2}\mathcal{A}(2s-2)+\cdots+\binom{s}{s}\mathcal{A}(s)\right]\\
        &+\binom{s}{2}\left[\binom{s}{1}\mathcal{A}(2s-2)+\binom{s}{2}\mathcal{A}(2s-3)+\cdots+\binom{s}{s}\mathcal{A}(s-1)\right]\\
        &\vdots\\
        &+\binom{s}{s}\left[\binom{s}{1}\mathcal{A}(s)+\binom{s}{2}\mathcal{A}(s-1)+\cdots+\binom{s}{s}\mathcal{A}(1)\right].
    \end{align*}
    Using Lemma \ref{lemma 3.1} we get 
    \begin{align*}
\mathcal{R}'(s)&=\sum_{i=1}^{s}\binom{s}{i}\frac{(p-2)^{s+1-i}}{p-1}\left[(p-1)^s-(p-2)^s\right]+\sum_{i=1}^{s}\binom{s}{i}(-1)^i\frac{p-2}{p-1}\\
    &=\frac{p-2}{p-1}\left[(p-1)^s-(p-2)^s\right]^2-\frac{p-2}{p-1}\\
    &=\frac{p-2}{p-1}\left[((p-1)^s-(p-2)^s)^2-1\right].
    \end{align*}
    Hence from \eqref{july17} we get the following recurrence,
    \begin{align*}
        \mathcal{T}(s+1)=(p-1)^{2s}+p\mathcal{T}(s)+\frac{p(p-2)}{p-1}\left[((p-1)^s-(p-2)^s)^2-1\right]-2\mathcal{S}(s),
    \end{align*}
    which using Lemma \ref{S(s)1} takes the form 
    \begin{align*}
        \mathcal{T}(s+1)=p\mathcal{T}(s)+(p-1)^{2s}A_1+(p-2)^{2s-2}A_2+(p-1)^s(p-2)^sA_3+A_4,
    \end{align*}
    where $A_i$'s are constants defined as
    \begin{align*}
        A_1&=\frac{(p-1)^2-p}{p-1};\\
        A_2&=\frac{p^4-8p^3+20p^2-16p}{p-1};\\
        A_3&=\frac{-2(p^2-4p+2)}{p-1};\\
        A_4&=-\frac{p^2-4}{p-1}.
    \end{align*}
    Now we use iteration over $s$ in $\mathcal{T}(s)$ to find the expression for $\mathcal{T}(s+1)$. In particular, we get 
    \begin{align*}
        \mathcal{T}(s+1)&=p^s\mathcal{T}(1)\\
        &+A_1(p-1)^{2s}\left[1+\frac{p}{(p-1)^2}+\frac{p^2}{(p-1)^4}\cdots +\frac{p^{s-1}}{(p-1)^{2(s-1)}}\right]\\
        &+A_2\frac{(p-2)^{2s}}{(p-2)^2}\left[1+\frac{p}{(p-2)^2}+\frac{p^2}{(p-2)^4}\cdots +\frac{p^{s-1}}{(p-2)^{2(s-1)}}\right]\\
        &+A_3(p-1)^s(p-2)^s\left[1+\frac{p}{(p-1)(p-2)}+\cdots+\frac{p^{s-1}}{(p-1)^{s-1}(p-2)^{s-1}}\right]\\
        &+A_4\left[1+p+\cdots+p^{s-1}\right]\\
        &=p^s\mathcal{T}(1)+A_1(p-1)^2\frac{(p-1)^{2s}-p^s}{(p-1)^2-p}+A_2\frac{(p-2)^{2s}-p^s}{(p-2)^2-p}\\
        &+A_3(p-1)(p-2)\frac{(p-1)^s(p-2)^s-p^s}{(p-1)(p-2)-p}+A_4\frac{p^s-1}{p-1}.
    \end{align*}
     Here $\mathcal{T}(1)=p-1$. Putting everything together, we get
     \begin{align*}
          \mathcal{T}(s+1)&=p^s(p-1)+(p-1)^{2s+1}-p^s(p-1)-2(p-2)((p-1)^s(p-2)^s-p^s)\\
          &+f(p,s+1).
     \end{align*}
    Finally replacing $s$ by $s-1$, we complete the proof.
\end{proof}
\begin{remark}\label{remark}
  It can be easily seen that for any positive integer $s$, $f(p,s)$ is an integer which can be observed by comparing the factors in the denominator and the numerator in the expression of $f(p,s)$.
\end{remark}
     In \cite[Lemma 4]{BB1} we calculated the number of $\mathbb{F}_p$-points on the surface,
     \begin{align*}
     (x-1)(y-1)(z-1)(1-xyz)=uxyz,~u\neq 0.
\end{align*} 
 where we used a proof of P. Deligne, which uses deep algebraic geometric method.
 In this article, we are interested in same congruence equations but with arbitrary number of variables. Instead, here we only use elementary method and estimates on character sums to count the number of points.

     \begin{lemma}\label{lemma M(s)}
     Let $p$ be an odd prime, $s$ be a positive integer and $\mathcal{M}(s)$ be the cardinality of the set  
\begin{eqnarray*}
 \{(x_1,...,x_{s},y_1,...,y_{s})\in\mathbb{F}_p^{2s}|x_1\cdots\ x_{s}\equiv y_1\cdots y_{s}\equiv 1\bmod p,
 \\ \prod_{i=1}^{s}(x_i-1)\equiv \prod_{i=1}^{s}(y_i-1)\not\equiv 0\bmod p,1\leq x_i, y_i\leq p-1\}.
\end{eqnarray*}
Then $\mathcal{M}(s)$ satisfies the asymptotic formula 
\begin{align*}
\mathcal{M}(s)&=\frac{(p-2)^{2s}}{(p-1)^3}+O(p^{s-\frac{1}{2}}).
\end{align*}
     \end{lemma}
     \begin{proof}
     Using the  properties of character sum, the given expression can be re-written as 
      \begin{align*}
&\mathop{\sum_{X_s}\sum_{Y_s}}_{\substack{\prod_{i=1}^{s}x_i\equiv\prod_{i=1}^{s}y_i \equiv 1 \bmod  p\\ 
 \prod_{i=1}^{s}(x_i-1)\equiv  \prod_{i=1}^{s}(y_i-1)\not\equiv 0\bmod p}}1\\
 &=\frac{1}{(p-1)^3}\sum_{X_s}\sum_{Y_s}\sum_{\chi_1,~\chi_2,~\chi_3\in\hat{\mathbb{F}}_p}\chi_1\left(\prod_{i=1}^{s}x_i\right)\chi_2\left(\prod_{i=1}^{s}y_i\right)\chi_3\left( \prod_{i=1}^{s}(x_i-1) \prod_{i=1}^{s}\overline{(y_i-1)}\right)\\
 &=\frac{1}{(p-1)^3}\sum_{\chi_1,~\chi_2,~\chi_3\in\hat{\mathbb{F}}_p}J(x_1,\chi_3)^sJ(\chi_2,\overline{\chi}_3)^s,
      \end{align*}
      where $J(\cdot,\cdot)$ is the classical Jacobi sum as defined in \eqref{jacobisum}. If we split out the trivial part and write the rest part in terms of Gauss sums using \eqref{Gauss-Jacobi}, we get
      \begin{align}\label{july11}
&\mathop{\sum_{X_s}\sum_{Y_s}}_{\substack{\prod_{i=1}^{s}x_i\equiv\prod_{i=1}^{s}y_i \equiv 1 \bmod  p\\ 
 \prod_{i=1}^{s}(x_i-1)\equiv  \prod_{i=1}^{s}(y_i-1)\not\equiv 0\bmod p}}1\notag\\
 &=\frac{(p-2)^{2s}}{(p-1)^3}
+\frac{1}{(p-1)^3}\sum_{\substack{\chi_1,~\chi_2,~\chi_3\neq 1\\
\chi_1\chi_3\neq 1,\chi_2\overline{\chi}_3\neq 1}}\frac{G(\chi_1)^sG(\chi_3)^sG(\chi_2)^sG(\overline{\chi}_3)^s}{G(\chi_1\chi_3)^sG(\chi_2\overline{\chi}_3)^s}+O(p^{s-1}).
      \end{align}
      We already have 
      \begin{align*}
      G(\chi_3)G(\overline{\chi}_3)=p{\chi_3(-1)}.
      \end{align*}
Now we use Proposition \ref{prop1} to get a bound for  
\begin{align}\label{p^s}
\sum_{\substack{\chi_1,~\chi_2,~\chi_3\neq 1\\
\chi_1\chi_3\neq 1,\chi_2\overline{\chi}_3\neq 1}}\frac{G(\chi_1)^sG(\chi_3)^sG(\chi_2)^sG(\overline{\chi}_3)^s}{G(\chi_1\chi_3)^sG(\chi_2\overline{\chi}_3)^s}=p^s\sum_{\substack{\chi_1,~\chi_2,~\chi_3\neq 1\\
\chi_1\chi_3\neq 1,\chi_2\overline{\chi}_3\neq 1}}\frac{\chi_3(-1)^sG(\chi_1)^sG(\chi_2)^s}{G(\chi_1\chi_3)^sG(\chi_2\overline{\chi}_3)^s}.
\end{align}
In \eqref{phi_m(x)}, we choose $m=1$, ${\bf{t_1}}=(1,1,-1)$, ${\bf{t}}_i=(1,1,1)$, for $i=2,3,4;$
$\eta_i=\chi_0$ for $i=1,2,3,4$. Take ${\bf{a}}_i$ as the ${i}$-th row of the matrix
\begin{align*}
\left[\begin{matrix}
1&0&0\\
0&1&0\\
1&0&1\\
0&1&-1
\end{matrix}\right].
\end{align*}
Note that, no two rows of the matrix are proportional which is satisfying the required condition for Proposition 3.1. For more details, see \cite[Theorem 1]{arl}.
This gives 
\begin{align*}
\Phi_1(\chi)=\left(p^{-1/2}\chi_3(-1)G(\chi_1),p^{-1/2}G(\chi_2),p^{-1/2}G(\chi_1\chi_3),p^{-1/2}G(\chi_2\overline{\chi}_3)\right).
\end{align*}
Now we choose $r=3$ and ${\bf{c}}=(s,s,-s,-s)$ in \eqref{Sigma} to get 
\begin{align*}
\sum_{\chi\in S_1}\Lambda_{{\bf{c}}}(\Phi_1(\chi))=\sum_{\substack{\chi_1,~\chi_2,~\chi_3\neq 1\\
\chi_1\chi_3\neq 1,\chi_2\overline{\chi}_3\neq 1}}\frac{\chi_3(-1)^sG(\chi_1)^sG(\chi_2)^s}{G(\chi_1\chi_3)^sG(\chi_2\overline{\chi}_3)^s}.
\end{align*}
Hence using Proposition \ref{prop1}, we get 
\begin{align}\label{july12}
\frac{1}{(p-1)^3}\sum_{\substack{\chi_1,~\chi_2,~\chi_3\neq 1\\
\chi_1\chi_3\neq 1,\chi_2\overline{\chi}_3\neq 1}}\frac{\chi_3(-1)^sG(\chi_1)^sG(\chi_2)^s}{G(\chi_1\chi_3)^sG(\chi_2\overline{\chi}_3)^s}=O(p^{-1/2}).
\end{align}
Finally combining \eqref{july11}, \eqref{p^s} and \eqref{july12}, we prove the lemma. 

     \end{proof}

    \section{Estimates for multivariable character sums.}
In this section, we prove a few lemmas which we will apply in the proof of the main theorem in section 5. The proof of the lemma follows the same approach as \cite[Lemma 1]{BB1} and \cite[Lemma 2.1]{zhang-xing} but is applicable for general $s$.
\begin{lemma}\label{w1}
 Let $p$ be an odd prime and $\chi$ be any Dirichlet character mod $p$. Then we have the identity
 \begin{eqnarray}
&&\sum_{m=1}^{p-1}\left|\sum_{X_s}
\chi(\prod_{i=1}^sx_i) e\left(\frac{\sum_{i=1}^sx_i+m\prod_{i=1}^s\overline{x_i}}{p}\right)\right|^2\nonumber\\
 &&=\begin{cases}
     p^{s+1}-2p^{s}-p^{s-1}-\cdots-p, & \mbox{if $\chi\neq\chi_0$};\\
    p^{s+1}-p^{s}-p^{s-1}-\cdots-p, & \mbox{if $\chi=\chi_0$}.\nonumber
   \end{cases}
\end{eqnarray}
\end{lemma}
\begin{proof}
We first note the following identity
\begin{equation}\label{july-1}
 \sum_{m=0}^{p-1}e\left(\frac{nm}{p}\right)=\begin{cases}
     p, & \mbox{if $p|n$};\\
    0, & \mbox{if $p\nmid n$}.
   \end{cases}
\end{equation}
Using the above identity, we have
 \begin{eqnarray}\label{k2}
&&\sum_{m=0}^{p-1}\left|\sum_{X_s}\chi(\prod_{i=1}^{s}x_i)e\left(\frac{\sum_{i=1}^{s}x_i+m\prod_{i=1}^{s}\overline{x_i}}{p}\right)\right|^2\notag\\
 &&=\sum_{X_s}\sum_{Y_s}\chi(\prod_{i=1}^{s}x_i\overline{y_i})\sum_{m=0}^{p-1}e\left(\frac{\sum_{i=1}^{s}x_i-\sum_{i=1}^{s}y_i+m(\prod_{i=1}^{s}
 \overline{x_i}-\prod_{i=1}^{s}\overline{y_i})}{p}\right)\notag\\
 &&=\sum_{X_s}\sum_{Y_s}\chi(\prod_{i=1}^{s}x_i)e\left(\frac{\sum_{i=1}^{s}y_i(x_i-1)}{p}\right)\times\sum_{m=0}^{p-1}e\left(\frac{m\prod_{i=1}^{s}\overline{y_i}
 (\prod_{i=1}^{s}\overline{x_i}-1)}{p}\right)\notag\\
 &&=p\mathop{\sum_{X_s}}_{\prod_{i=1}^{s}x_i\equiv1\bmod p}
 \chi(\prod_{i=1}^{s}x_i)\sum_{Y_s}e\left(\frac{\sum_{i=1}^{s}y_i(x_i-1)}{p}\right)\notag\\
 &&=p\sum_{Y_s}1+\binom{s}{0}p\mathop{\sum_{x_1=2}^{p-1}\cdots\sum_{x_{s}=2}^{p-1} \sum_{Y_s}}_{\prod_{i=1}^{s}x_i \equiv 1 \bmod p}e\left(\frac{\sum_{i=1}^{s}y_i(x_i-1)}{p}\right)\notag\\&&+\binom{s}{1}p\mathop{\sum_{x_1=2}^{p-1}\cdots\sum_{x_{s-1}=2}^{p-1} \sum_{Y_s}}_{\prod_{i=1}^{s-1}x_i \equiv 1 \bmod p}e\left(\frac{\sum_{i=1}^{s-1}y_i(x_i-1)}{p}
 \right)\notag\\
 &&+\binom{s}{2}p\mathop{\sum_{x_1=2}^{p-1}\cdots\sum_{x_{s-2}=2}^{p-1} \sum_{Y_s}}_{\prod_{i=1}^{s-2}x_i \equiv 1 \bmod p}e\left(\frac{\sum_{i=1}^{s-2}y_i(x_i-1)}{p}\right)\notag\\
 && +\cdots+\binom{s}{s-2}p\mathop{\sum_{x_1=2}^{p-1}\sum_{x_{2}=2}^{p-1} \sum_{Y_s}}_{\prod_{i=1}^{2}x_i \equiv 1 \bmod p}e\left(\frac{\sum_{i=1}^{2}y_i(x_i-1)}{p}\right)\notag\\
 &&=p(p-1)^s+\binom{s}{0}p(-1)^s\mathcal{A}(s)+\binom{s}{1}p(p-1)(-1)^{s-1}\mathcal{A}(s-1)\notag\\&&+\binom{s}{2}p(p-1)^2(-1)^{s-2}\mathcal{A}(s-2)+\cdots+\binom{s}{s-2}p(p-1)^{s-2}(-1)^2\mathcal{A}(2)\notag,
 \end{eqnarray}
 where $A(s)$ is same as in Lemma \ref{lemma 3.1}. Now using the formula for $\mathcal{A}(s)$, we get 
 \begin{align*}
  &\sum_{m=0}^{p-1}\left|\sum_{X_s}\chi(\prod_{i=1}^{s}x_i)e\left(\frac{\sum_{i=1}^{s}x_i+m\prod_{i=1}^{s}\overline{x_i}}{p}\right)\right|^2\\
  &=p\left[(p-1)^s+\binom{s}{0}(-1)^s\mathcal{A}(s)+\binom{s}{1}(p-1)(-1)^{s-1}\mathcal{A}(s-1)\right.\\
 &+ \left.\binom{s}{2}(p-1)^2(-1)^{s-2}\mathcal{A}(s-2)+\cdots+\binom{s}{s-2}(p-1)^{s-2}(-1)^{s-2}\mathcal{A}(2)\right]\\
 &=p(p-1)^s+\frac{(-1)^s}{p-1}p\left[\binom{s}{0}(p-2)^s+\binom{s}{1}(1-p)(p-2)^{s-1}\right.\\&\left.+\binom{s}{2}(p-2)^{s-2}(1-p)^2+\cdots+\binom{s}{s-2}(1-p)^{s-2}(p-2)^2\right]\\&+\frac{p-2}{p-1}p\left[\binom{s}{0}+\binom{s}{1}(p-1)+\binom{s}{2}(p-1)^2 +\cdots+\binom{s}{s-2}(p-1)^{s-2}\right]\\
 &=p(p-1)^s+\frac{1}{p-1}p+\frac{p-2}{p-1}p^{s+1}+(-1)^sp\left[s(1-p)^{s-2}(p-2)+(1-p)^{s-1}\right]\\
 &-(p-2)p\left[s(p-1)^{s-2}+(p-1)^{s-1}\right]\\
 &=p(p-1)^s+\frac{p^{s+1}-2p^s+1}{p-1}p-p(p-1)^{s}.\\
 &=p(p-1)^s+p^{s+1}-p(p^{s-1}+p^{s-2}+\cdots+1)-p(p-1)^s\\
 &=p^{s+1}-p^{s}-p^{s-1}-\cdots-1.
 \end{align*}
 Finally using the properties of Gauss sum in \eqref{july-1}, we complete the proof.
\end{proof}
\begin{lemma}\label{l4}
Let $p$ be an odd prime and $s$ be a positive integer. Then we have the identity
\begin{eqnarray*}
&\mathop{\sum_{X_s} \sum_{Y_s}}_{\substack{\prod_{i=1}^{s}x_i\equiv\prod_{i=1}^{s}y_i \bmod  p\\ \prod_{i=1}^{s}(x_i-1)\equiv \prod_{i=1}^{s}(y_i-1)\bmod p}}
 \chi_0\left((\prod_{i=1}^{s}\overline{x_i}-1)\prod_{i=1}^{s}(x_i-1)\right)\\
& =p^{s}+\frac{(p-2)^{2s+1}}{(p-1)^3}+O(p^{s-1/2}).
\end{eqnarray*}
\end{lemma}
\begin{proof}
 We have 
 \begin{eqnarray}\label{y1}
  &&\mathop{\sum_{X_s}\sum_{Y_s}}_{\substack{\prod_{i=1}^{s}x_i\equiv\prod_{i=1}^{s}y_i \bmod  p\\ \prod_{i=1}^{s}(x_i-1)\equiv \prod_{i=1}^{s}(y_i-1)\bmod p}}
\hspace{-1cm}\chi_0\left((\prod_{i=1}^{s}\overline{x_i}-1)\prod_{i=1}^{s}(x_i-1)\right)\notag\\
 &&=\mathop{\sum_{X_s}\sum_{Y_s}}_{\substack{\prod_{i=1}^{s}x_i\equiv\prod_{i=1}^{s}y_i \bmod  p\\ 
 \prod_{i=1}^{s}(x_i-1)\equiv \prod_{i=1}^{s}(y_i-1)\bmod p}}1-\mathop{\sum_{X_s}\sum_{Y_s}}_{\substack{\prod_{i=1}^{s}x_i\equiv\prod_{i=1}^{s}y_i \equiv 1 \bmod  p\\ 
 \prod_{i=1}^{s}(x_i-1)\equiv  \prod_{i=1}^{s}(y_i-1)\bmod p}}1\notag\\ 
 &&-\mathop{\sum_{X_s}\sum_{Y_s}}_{\substack{\prod_{i=1}^{s}x_i\equiv\prod_{i=1}^{s}y_i \bmod  p\\ 
 \prod_{i=1}^{s}(x_i-1)\equiv \prod_{i=1}^{s}(y_i-1)\equiv 0\bmod p}}\hspace{-.5cm}\chi_0 \left(\prod_{i=1}^{s}\overline{x_i}-1\right).
 \end{eqnarray}
 Notice that the first term in \eqref{y1} is exactly equal to $\mathcal{T}(s)$ in Lemma \ref{lemma T(s)}, where we deduced 
  \begin{align}\label{T(s)}
\mathcal{T}(s)=
(p-1)^{2s-1}-2(p-2)\left((p-1)^{s-1}(p-2)^{s-1}-p^{s-1}\right)+f(p,s),
\end{align}
where 
\begin{align*}
    f(p,s)=\frac{p(p-2)^2((p-2)^{2s-2}-p^{s-1})}{(p-1)^2}-\frac{(p^2-4)(p^{s-1}-1)}{(p-1)^2}.
\end{align*}
   Next, we have
    \begin{eqnarray}\label{june1}
&&\mathop{\sum_{X_s}\sum_{Y_s}}_{\substack{\prod_{i=1}^{s}x_i\equiv\prod_{i=1}^{s}y_i \equiv 1 \bmod  p\\ \prod_{i=1}^{s}(x_i-1)\equiv 
 \prod_{i=1}^{s}(y_i-1)\bmod p}}1
 = \left(\mathop{\sum_{X_s}}_{\substack{\prod_{i=1}^{s}x_i\equiv 1 \bmod  p,\\ ~\prod_{i=1}^{s}(x_i-1)\equiv 
 0\bmod p}}1\right)^2+\mathcal{M}(s),
  \end{eqnarray}
  where $\mathcal{M}(s)$ is same as in Lemma \ref{lemma M(s)}.
  Now 
  \begin{align*}
&\mathop{\sum_{X_s}}_{\substack{\prod_{i=1}^{s}x_i\equiv 1 \bmod  p,\\ ~\prod_{i=1}^{s}(x_i-1)\equiv 
 0\bmod p}}1\\
 &=\binom{s}{1}\mathcal{A}(s-1)+\binom{s}{2}\mathcal{A}(s-2)+\cdots+\binom{s}{s-2}\mathcal{A}(2)+\binom{s}{s-1}\mathcal{A}(1),
  \end{align*}
  which using 
  Lemma \ref{lemma 3.1} becomes
  \begin{align}\label{june2}
\mathop{\sum_{X_s}}_{\substack{\prod_{i=1}^{s}x_i\equiv 1 \bmod  p,\\ ~\prod_{i=1}^{s}(x_i-1)\equiv 
 0\bmod p}}1=\frac{1}{p-1}\left[(p-1)^s-(p-2)^s-1\right]-\frac{p-2}{p-1}((-1)^s+1).
  \end{align}

Finally, from the properties of reduced residue system mod $p$ we have
  \begin{align}\label{r1}
\mathop{\sum_{X_s}\sum_{Y_s}}_{\substack{\prod_{i=1}^{s}x_i\equiv\prod_{i=1}^{s}y_i \bmod  p\\ \prod_{i=1}^{s}(x_i-1)\equiv \prod_{i=1}^{s}(y_i-1)\equiv 0\bmod p}}\chi_0 
  \left(\overline{x_1\cdots x_s}-1\right)
  =\sum_{u=2}^{p-1}\left(\mathop{\sum_{X_s}}_{\substack{\prod_{i=1}^{s}x_i\equiv u\bmod  p\\ \prod_{i=1}^{s}(x_i-1)\equiv 0\bmod p}}1\right)^2.
  \end{align}
  Hence using Lemma \ref{lemma 3.2} we get
  \begin{align*}
&\mathop{\sum_{X_s}}_{\substack{\prod_{i=1}^{s}x_i\equiv u\bmod  p\\ \prod_{i=1}^{s}(x_i-1)\equiv 0\bmod p}}1\\&=\binom{s}{1}\mathcal{A}_u(s-1)+\binom{s}{2}\mathcal{A}_u(s-2)+\cdots+\binom{s}{s-2}\mathcal{A}_{u}(2)+\binom{s}{s-1}\mathcal{A}_u(1)\\&=\frac{1}{p-1}\left[\binom{s}{1}(p-2)^{s-1}+\cdots+\binom{s}{s-2}(p-2)^2\right]\\
  &-\frac{1}{p-1}\left[\binom{s}{1}(-1)^{s-1}+\cdots+\binom{s}{s-2}(-1)^2\right]+s\\
  &=\frac{1}{p-1}\left[(p-1)^s-(p-2)^s-s(p-2)-1\right]\\&-\frac{1}{p-1}\left[(-1)^{s+1}+s-1\right]+s\\
  &=\frac{1}{p-1}\left[(p-1)^s-(p-2)^s-s(p-1)+(-1)^s\right]+s\\
  &=\frac{1}{p-1}\left[(p-1)^s-(p-2)^s+(-1)^s\right].
  \end{align*}
  Hence \eqref{r1} implies
  \begin{align}\label{june4}
&\mathop{\sum_{X_s}\sum_{Y_s}}_{\substack{\prod_{i=1}^{s}x_i\equiv\prod_{i=1}^{s}y_i \bmod  p\\ \prod_{i=1}^{s}(x_i-1)\equiv \prod_{i=1}^{s}(y_i-1)\equiv 0\bmod p}}\hspace{-1cm}\chi_0 
  \left(\overline{x_1\cdots x_s}-1\right)
  =\frac{p-2}{(p-1)^2}\left[(p-1)^s-(p-2)^s+(-1)^s\right]^2.
  \end{align}
Combining \eqref{y1}, \eqref{T(s)}, \eqref{june1}, \eqref{june2}, \eqref{june4} and using the estimate for $\mathcal{M}(s)$ in Lemma \ref{lemma M(s)}, we obtain the required result.
  \end{proof}
  \begin{lemma}\label{sh}
  Let $p$ be an odd prime. Then we have
  \begin{eqnarray*}
  \mathop{{\sum_{x_1=2}^{p-1}\cdots\sum_{x_s=2}^{p-1}\sum_{y_1=2}^{p-1}\cdots \sum_{y_s=2}^{p-1}}}_{\prod_{i=1}^{s}x_i\equiv \prod_{i=1}^{s}y_i\bmod p}\chi_0(\prod_{i=1}^{s}x_i-1)
  =(p-2)\left(\frac{(p-2)^s-(-1)^s}{p-1}\right)^2
  \end{eqnarray*}
   \end{lemma}
   \begin{proof}
   We have
   \begin{eqnarray}\label{ab1}
   \mathop{{\sum_{x_1=2}^{p-1}\cdots\sum_{x_s=2}^{p-1}\sum_{y_1=2}^{p-1}
   \cdots\sum_{y_s=2}^{p-1}}}_{\prod_{i=1}^{s}x_i\equiv \prod_{i=1}^{s}y_i\bmod p}\chi_0(\prod_{i=1}^{s}x_i-1)
  =\sum_{u=2}^{p-1}\left(\mathop{{\sum_{x_1=2}^{p-1}\cdots\sum_{x_s=2}^{p-1}}}_{\prod_{i=1}^{s}x_i\equiv u\bmod p}1\right)^2.\notag
   \end{eqnarray}
   Hence using Lemma \ref{lemma 3.2} we get
   \begin{align*}
       \mathop{{\sum_{x_1=2}^{p-1}\cdots\sum_{x_s=2}^{p-1}\sum_{y_1=2}^{p-1}
   \cdots\sum_{y_s=2}^{p-1}}}_{\prod_{i=1}^{s}x_i\equiv \prod_{i=1}^{s}y_i\bmod p}\chi_0(\prod_{i=1}^{s}x_i-1)
   =
(p-2)\left(\frac{(p-2)^s-(-1)^s}{p-1}\right)^2.
   \end{align*}
   \end{proof}
  
 
\begin{lemma}\label{y5}
 Let $p$ be an odd prime and $\chi$ be any character mod $p$. Then for primitive character $\chi_1$ mod $p$, we have the identity
 \begin{align*}
  &\left|\sum_{m=1}^{p-1}\chi_1(m)\left|\sum_{X}\chi(\prod_{i=1}^{s}x_i)e\left(\frac{\sum_{i=1}^{s}x_i+m\prod_{i=1}^{s}\overline{x_i}}{p}\right)\right|^2\right|^2\\
 &=p^{s+1}\left|\sum_{X}\chi(\prod_{i=1}^{s}x_i)\overline{\chi_1}((\prod_{i=1}^{s}\overline{x_i}-1)\prod_{i=1}^{s}(x_i-1))\right|^2.
 \end{align*}
 \end{lemma}
 \begin{proof}
 The proof proceeds along similar lines to the proof  of \cite[Lemma 2.3]{zhang-xing}. So we omit the details for reasons of brevity. 
 \end{proof}

\section{Proof of Theorem \ref{t1}}
We now have all the ingredients to prove our main theorem.
\begin{proof}[Proof of Theorem \ref{t1}]
From the orthogonality property of characters mod $p$ we have
\begin{eqnarray}\label{y11}
&&\sum_{\chi\hspace{-.2cm} \mod p}\sum_{\chi_1 \hspace{-.2cm} \mod p} \left|\sum_{m=1}^{p-1}\chi_1(m)\left|\sum_{X_s} \chi(\prod_{i=1}^{s}x_i)e\left(\frac{\sum_{i=1}^{s}x_i+m\prod_{i=1}^{s} \overline{x_i}}{p}\right)\right|^2\right|^2\notag\\
&&=(p-1)\sum_{m=1}^{p-1}\sum_{\chi \hspace{-.2cm} \mod p}\left|\sum_{X_s}\chi(\prod_{i=1}^{s}x_i)e\left(\frac{\sum_{i=1}^{s}x_i+m\prod_{i=1}^{s}\overline{x_i}}{p}\right)\right|^4.
\end{eqnarray}
Also we have
\begin{eqnarray}\label{h1}
&&\sum_{\chi \hspace{-.2cm} \mod p}\sum_{\chi_1 \hspace{-.2cm} \mod p} \left|\sum_{m=1}^{p-1}\chi_1(m)\left|\sum_{X} \chi(\prod_{i=1}^{s}x_i)e\left(\frac{\sum_{i=1}^{s}x_i+m\prod_{i=1}^{s}\overline{x_i}}{p}\right)\right|^2\right|^2\notag\\
&&=\sum_{\chi \hspace{-.2cm} \mod p}\sum_{\substack{\chi_1 \hspace{-.2cm} \mod p\\ \chi_1 \neq \chi_0}} \left|\sum_{m=1}^{p-1}\chi_1(m)\left|\sum_{X_s}\chi(\prod_{i=1}^{s}x_i)e\left(\frac{\sum_{i=1}^{s}x_i
	+m\prod_{i=1}^{s}\overline{x_i}}{p}\right)\right|^2\right|^2\notag\\
&&\hspace{.5cm}+\sum_{\chi \hspace{-.2cm} \mod p}\left|\sum_{m=1}^{p-1}\left|\sum_{X_s} \chi(\prod_{i=1}^{s}x_i)e\left(\frac{\sum_{i=1}^{s}x_i+m\prod_{i=1}^{s}\overline{x_i}}{p}\right)\right|^2\right|^2.
\end{eqnarray}
Using Lemma \ref{w1} we obtain
\begin{eqnarray}\label{h2}
&&\sum_{\chi \hspace{-.2cm} \mod p}\left|\sum_{m=1}^{p-1}\left|\sum_{X_s} \chi(\prod_{i=1}^{s}x_i)e\left(\frac{\sum_{i=1}^{s}x_i+m\prod_{i=1}^{s}\overline{x_i}}{p}\right)\right|^2\right|^2\notag\\
&&=(p-2)(p^{s+1}-2p^{s}-p^{s-1}-\cdots--p)^2+(p^{s+1}-p^{s}-p^{s-1}-\cdots--p)^2\notag\\
&&=(p-1)(p^{2(s+1)}+O(p^{2s+1})).
\end{eqnarray}
Now using Lemmas \ref{l4}, \ref{sh} and \ref{y5} we have
\begin{eqnarray}\label{h3}
&&\sum_{\chi \hspace{-.2cm} \mod p}\sum_{\substack{\chi_1 \hspace{-.2cm} \mod p\\ \chi_1 \neq \chi_0}} \left|\sum_{m=1}^{p-1}\chi_1(m)\left|\sum_{X_s}\chi(\prod_{i=1}^{s}x_i)
e\left(\frac{\sum_{i=1}^{s}x_i+m\prod_{i=1}^{s}\overline{x_i}}{p}\right)\right|^2\right|^2\notag\\
&&=p^{s+1}\sum_{\chi \hspace{-.2cm} \mod p}\sum_{\substack{\chi_1 \hspace{-.2cm} \mod p\\ \chi_1 \neq \chi_0}}\left|\sum_{X_s}\chi(\prod_{i=1}^{s}x_i)\overline{\chi_1}((\prod_{i=1}^{s}\overline{x_i}-1)
\prod_{i=1}^{s}(x_i-1))\right|^2\notag\\
&&=p^{s+1}\sum_{\chi \hspace{-.2cm} \mod p}\sum_{\substack{\chi_1 \hspace{-.2cm} \mod p}}\left|\sum_{X_s}\chi(\prod_{i=1}^{s}x_i)\overline{\chi_1}((\prod_{i=1}^{s}\overline{x_i}-1)
\prod_{i=1}^{s}(x_i-1))\right|^2\notag\\
&&-p^{s+1}\sum_{\chi \hspace{-.2cm} \mod p}\left|\sum_{X_s}\chi(\prod_{i=1}^{s}x_i)\overline{\chi_0}((\prod_{i=1}^{s}\overline{x_i}-1)
\prod_{i=1}^{s}(x_i-1))\right|^2\notag\\
&&=p^{s+1}(p-1)^2\mathop{{\sum_{X_s}\sum_{Y_s}}}
_{\substack{\prod_{i=1}^{s}x_i\equiv\prod_{i=1}^{s}y_i\bmod p \\ \prod_{i=1}^{s}(x_i-1)\equiv \prod_{i=1}^{s}(y_i-1)\bmod p}}\chi_0((\prod_{i=1}^{s}\overline{x_i}-1)\prod_{i=1}^{s}(x_i-1))\notag\\
&&-p^{s+1}(p-1)\mathop{{\sum_{x_1=2}^{p-1}\cdots\sum_{x_s=2}^{p-1}\sum_{y_1=2}^{p-1}\cdots\sum_{y_s=2}^{p-1}}}
_{\prod_{i=1}^{s}x_i\equiv \prod_{i=1}^{s}y_i\bmod p}\chi_0(\prod_{i=1}^{s}x_i-1)\notag\\
&&=p^{s+1}(p-1)\left[p^{s+1}+\frac{(p-2)^{2s+1}}{(p-1)^2}+O(p^{s+1/2})\right]\notag\\&&-p^{s+1}(p-1)(p-2)\left[\frac{(p-2)^s}{(p-1)}-(-1)^s\right]^2\notag\\
&&=(p-1)\left(p^{2s+2}+O(p^{2s+3/2})\right)
\end{eqnarray}
Hence from \eqref{h1}, \eqref{h2} and \eqref{h3} we obtain
\begin{eqnarray}\label{y10}
&&\sum_{\chi \hspace{-.2cm} \mod p}\sum_{\chi_1 \hspace{-.2cm} \mod p} \left|\sum_{m=1}^{p-1}\chi_1(m)\left|\sum_{X} \chi(\prod_{i=1}^{s}x_i)e\left(\frac{\sum_{i=1}^{s}x_i+m\prod_{i=1}^{s}
	\overline{x_i}}{p}\right)\right|^2\right|^2\notag\\
&&=(p-1)\left((p^{2s+2}+O(p^{2s+3/2})+(p^{2s+2}+O(p^{2s+1})\right)\notag\\
&&=(p-1)\left(2p^{2s+2}+O(p^{2s+3/2})\right).
\end{eqnarray}
Using \eqref{y11} and \eqref{y10} we complete the proof of our main theorem.
\end{proof}
\section{Comments}
Note that, putting $s=2,3,4$ in Theorem \ref{t1}, we can deduce the main results in \cite{BB, BB1, zhang, zhang-xing}. For $s=3, 4$; Theorem \ref{t1} gives weaker results than \cite{BB} and \cite{BB1}. In \cite{BB} ans \cite{BB1} we followed different approaches to calculate the number of $\mathbb{F}_p$-points on the surface
\begin{align*}
(x_1-1)\cdots(x_s-1)(1-x_1\cdots x_s)\equiv u x_1\cdots x_s \bmod p.
\end{align*} 
For $s=3$, we use the absolute irreducibility property of the above surface and for $s=4$, we used a result of P. Deligne \cite[Lemma 4]{BB1}. Instead in this article we followed properties of Jacobi sums and a result of A. R-Le\'{o}n \cite{arl} which help us to prove Lemma 4.6, which gives a different approach to study such surface for general number of variable. At the same time the point counting in Section 4 plays a very crucial role in generalizing our result. Any improvement of our result in Theorem \ref{t1} is directly related to the improvement of the estimate in Lemma 4.6. Our theorem is important as it gives advancement to the known methods. The result is more general and establish asymptotic formula for \eqref{k1} of any dimension for the hyper-Kloosterman sum.
\par
\noindent 
\section{Acknowledgements:}
\par During the preparation of this article N.B. was supported by the National Board of Higher Mathematics post-doctoral fellowship (No.:~0204/3/2021/R\&D-II/7363) and A.H. would like to thank CSIR, Government of India for financial support in the form of a Senior Research Fellowship (No.: 09/934(0016)/2019-EMR-I). \par Also we would like to thank Antonio Rojas-Le\'{o}n for some helpful discussions and very important suggestion to use his recent result \cite[proposition 1]{arl} in Lemma \ref{lemma M(s)}.

\end{document}